\documentclass[a4paper,12pt]{amsart}

\usepackage{fullpage}
\usepackage{amsmath,amssymb,graphicx,psfrag}
\usepackage[T1]{fontenc}
\usepackage{amsmath,amssymb,ifthen}
\usepackage{color}
\usepackage{moreverb}
\usepackage{enumitem}
\usepackage{pgfplots}
\usepackage{adjustbox}
\usepackage{algorithmic}
\usepackage{mathtools}
\usepackage{IEEEtrantools}
\newcommand{\set}[3][\big]{#1\{#2\,:\,#3#1\}}

\newcounter{statement}
\newenvironment{statement}[2][!]{%
	\vskip3mm
	\hrule
	\hrule
	\hrule
	\vskip1mm
	\noindent%
	\refstepcounter{statement}%
	\bf#2~\thestatement%
	\ifthenelse{\equal{#1}{!}}{.\ }{~(#1).\ }%
	\it%
}{%
	\vskip1mm
	\hrule
	\hrule
	\hrule
	\vskip2mm
}

\newenvironment{theorem}[1][!]{\begin{statement}[#1]{Theorem}}{\end{statement}}
\newenvironment{lemma}[1][!]{\begin{statement}[#1]{Lemma}}{\end{statement}}

\newenvironment{remark}[1][!]{\begin{statement}[#1]{Remark}}{\end{statement}}

\newenvironment{definition}[1][!]{\begin{statement}[#1]{Definition}}{\end{statement}}

\renewcommand{\subsection}[1]{%
	\vskip2mm
	\refstepcounter{subsection}%
	{\bf\arabic{section}.\arabic{subsection}.~#1.~}%
}

\newcommand{\R}{\mathbb{R}}

\renewcommand{\S}{\mathbb{S}}
\newcommand{\EE}{\mathcal{E}}

\newcommand{\DD}{\mathcal{D}}
\newcommand{\helical}{\mathfrak{h}}

\newcommand{\abs}[2][]{#1|\,#2\,#1|}

\newcommand{\norm}[3][]{#1\|{#2}#1\|_{#3}}
\newcommand{\skp}[4][]{#1\langle{#2},{#3}#1\rangle_{#4}}

\renewcommand{\d}[1]{\mathrm{\, d}#1}

\DeclareMathOperator{\curl}{curl}

\newcommand{\lex}{\ell_\mathrm{ex}}

\renewcommand{\vec}[1]{\boldsymbol{#1}}
\newcommand{\normalv}{\vec{n}}
\newcommand{\m}{\vec{m}}
\newcommand{\w}{\vec{w}}
\renewcommand{\u}{\vec{u}}
\newcommand{\heff}{\vec{H}_\mathrm{eff}}

\usepackage{fancyhdr}
\advance\footskip0.4cm
\advance\textheight-0.4cm
\calclayout
\newcommand*\patchAmsMathEnvironmentForLineno[1]{%
	\expandafter\let\csname old#1\expandafter\endcsname\csname #1\endcsname
	\expandafter\let\csname oldend#1\expandafter\endcsname\csname end#1\endcsname
	\renewenvironment{#1}%
	{\linenomath\csname old#1\endcsname}%
	{\csname oldend#1\endcsname\endlinenomath}}% 
\newcommand*\patchBothAmsMathEnvironmentsForLineno[1]{%
	\patchAmsMathEnvironmentForLineno{#1}%
	\patchAmsMathEnvironmentForLineno{#1*}}%
\AtBeginDocument{%
	\patchBothAmsMathEnvironmentsForLineno{equation}%
	\patchBothAmsMathEnvironmentsForLineno{align}%
	\patchBothAmsMathEnvironmentsForLineno{flalign}%
	\patchBothAmsMathEnvironmentsForLineno{alignat}%
	\patchBothAmsMathEnvironmentsForLineno{gather}%
	\patchBothAmsMathEnvironmentsForLineno{multline}%
}
\usepackage[mathlines]{lineno}

\title{Weak-Strong Uniqueness for the Landau--Lifshitz--Gilbert equation in Micromagnetics}

\author{Giovanni Di Fratta}
\address{TU Wien, Institute for Analysis and Scientific Computing, Wiedner Hauptstr.~8--10/E101/4, 1040 Wien, Austria}
\email{Giovanni.DiFratta@asc.tuwien.ac.at}

\author{Michael Innerberger}
\address{TU Wien, Institute for Analysis and Scientific Computing, Wiedner Hauptstr.~8--10/E101/4, 1040 Wien, Austria}
\email{Michael.Innerberger@asc.tuwien.ac.at \qquad\rm (corresponding author)}

\author{Dirk Praetorius}
\address{TU Wien, Institute for Analysis and Scientific Computing, Wiedner Hauptstr.~8--10/E101/4, 1040 Wien, Austria}
\email{Dirk.Praetorius@asc.tuwien.ac.at}

\thanks{{\bf Acknowledgment:} The authors acknowledge support through the Austrian Science Fund (FWF) through the doctoral school \emph{Dissipation and dispersion in nonlinear PDEs} (grant W1245) and the special research program \emph{Taming complexity in PDE systems} (grant SFB F65).}

\date{\today}

\begin{document}%%%%%%%%%%%%%%%%%%%%%%%%%%%%%%%%%%%%%%%%%%%%%%%%%%%%%%%%%%%%%%%%%%%
	
%%%%%%%%%%%%%%%%%%%%%%%%%%%%%%%%%%%%%%%%%%%%%%%%%%%%%%%%%%%%%%%%%%%%%%%%%%%%%%%%%%%
%%%%%%%%%%%%%%%%%%%%%%%%%%%%%%%%%%%%%%%%%%%%%%%%%%%%%%%%%%%%%%%%%%%%%%%%%%%%%%%%%%%
\begin{abstract}
We consider the time-dependent Landau--Lifshitz--Gilbert equation. We prove that each weak solution coincides with the (unique) strong solution, as long as the latter exists in time. Unlike available results in the literature, our analysis also includes the physically relevant lower-order terms like Zeeman contribution, anisotropy, stray field, and the Dzyaloshinskii--Moriya interaction (which accounts for the emergence of magnetic Skyrmions). Moreover, our proof gives a template on how to approach weak-strong uniqueness for even more complicated problems, where LLG is (nonlinearly) coupled to other (nonlinear) PDE systems.
\end{abstract}
\maketitle
%%%%%%%%%%%%%%%%%%%%%%%%%%%%%%%%%%%%%%%%%%%%%%%%%%%%%%%%%%%%%%%%%%%%%%%%%%%%%%%%%%%
%%%%%%%%%%%%%%%%%%%%%%%%%%%%%%%%%%%%%%%%%%%%%%%%%%%%%%%%%%%%%%%%%%%%%%%%%%%%%%%%%%%

%\clearpage
%!TEX root = weak-strong.tex

%%%%%%%%%%%%%%%%%%%%%%%%%%%%%%%%%%%%%%%%%%%%%%%%%%%%%%%%%%%%%%%%%%%%%%%%%%%%%%%%%%%
%%%%%%%%%%%%%%%%%%%%%%%%%%%%%%%%%%%%%%%%%%%%%%%%%%%%%%%%%%%%%%%%%%%%%%%%%%%%%%%%%%%
\section{Introduction}\label{sec:intro}
%%%%%%%%%%%%%%%%%%%%%%%%%%%%%%%%%%%%%%%%%%%%%%%%%%%%%%%%%%%%%%%%%%%%%%%%%%%%%%%%%%%
%%%%%%%%%%%%%%%%%%%%%%%%%%%%%%%%%%%%%%%%%%%%%%%%%%%%%%%%%%%%%%%%%%%%%%%%%%%%%%%%%%%

%%%%%%%%%%%%%%%%%%%%%%%%%%%%%%%%%%%%%%%%%%%%%%%%%%%%%%%%%%%%%%%%%%%%%%%%%%%%%%%%%%%
%\subsection{LLG equation \& weak-strong uniqueness}\label{subsec:LLG}
%%%%%%%%%%%%%%%%%%%%%%%%%%%%%%%%%%%%%%%%%%%%%%%%%%%%%%%%%%%%%%%%%%%%%%%%%%%%%%%%%%%

The Landau--Lifshitz--Gilbert (LLG) equation is the well-accepted PDE model to describe magnetization dynamics of a ferromagnetic body $\Omega \subset \R^3$ in terms of the sought magnetization $\m : \Omega \to \S^2 := \set{x \in \R^3}{|x| = 1}$.
The time-dependent LLG equation poses several challenges like nonlinearities, a nonconvex pointwise constraint, an intrinsic energy law, which resembles the one of a gradient flow and combines conservative and dissipative effects, and the possible coupling with other PDEs, e.g., the Maxwell equations. 
On the one hand, weak solutions exist globally in time, but may be nonunique; see~\cite{AS92} for the seminal contribution in this direction. On the other hand, strong solutions to LLG exist only locally in time and under severe assumptions on the initial condition (see \cite{cf01a,cf01b,melcher12, ft17}; see also \cite{prohl2001computational}) .

One common question for solutions of PDEs is whether existing smooth and weak solutions coincide, rather than coexist.
Such weak-strong uniqueness results are ubiquitous in the literature for various PDE models, e.g., for the Navier--Stokes equations \cite{constantin89, Op18FM}. For LLG, however, weak-strong uniqueness has only been investigated recently in~\cite{DS14}. In the latter paper, the analysis focuses on a simplified setting with $\Omega = \R^3$ (i.e., possible boundary conditions are neglected) and the so-called effective field, which drives the evolution of $\m$, consists only of the leading-order exchange contribution. In particular, physically relevant lower-order terms like Zeeman field, anisotropy, stray field, and Dzyaloshinskii--Moriya interaction are excluded. The very recent preprint~\cite{kks19} considers weak-strong uniqueness for a more involved magnetoviscoelastic model, but restricts to a simplified setting for thin magnetic films, where $\Omega \subset \R^2$.

The purpose of the present paper is twofold: First, we prove weak-strong uniqueness for some physical relevant 3D setting of LLG, where $\Omega \subset \R^3$ is the bounded domain of the ferromagnet and where we account for all standard lower-order contributions to the effective field.
Second, the analysis of existing weak-strong uniqueness results in~\cite{DS14,kks19} involves much tedious algebra.
While our concept of proof is closely related to that of the seminal work~\cite{DS14}, we believe that our proof is more concise. In particular, our proof gives a template on how to approach weak-strong uniqueness for even more complicated problems, where LLG is (nonlinearly) coupled to other (nonlinear) PDE systems.

We note that weak-strong uniqueness results are of particular interest for the numerical integration of LLG systems: Available unconditionally convergent integrators converge (weakly on subsequences of the computed discrete solutions) towards a weak solution of LLG; see, e.g.,~\cite{MR2379897,MR2257110} for some seminal works on plain LLG or~\cite{MR3955721,MR3253344,tps2} for some coupled LLG systems. Weak-strong uniqueness thus implies that all these integrators will converge towards the same limit (even for the full sequence of computed solutions), at least as long as a strong solution exists.

\bigskip

%%%%%%%%%%%%%%%%%%%%%%%%%%%%%%%%%%%%%%%%%%%%%%%%%%%%%%%%%%%%%%%%%%%%%%%%%%%%%%%%%%%
%\subsection{Outline}\label{subsec:outline}
%%%%%%%%%%%%%%%%%%%%%%%%%%%%%%%%%%%%%%%%%%%%%%%%%%%%%%%%%%%%%%%%%%%%%%%%%%%%%%%%%%%
{\bf Outline.}
In Section~\ref{sec:result}, we give a thorough statement of LLG, formulate the notion of strong and weak solutions, and state our main result (Theorem~\ref{theorem:main}) on weak-strong uniqueness for LLG on a bounded Lipschitz domain $\Omega \subset \R^3$.
In Section~\ref{sec:auxiliary}, we reformulate  the LLG equation in terms of a helicity functional, which combines exchange energy and Dzyaloshinskii--Moriya interaction. The original idea for this is well-known in topological fluid dynamics~\cite{moffatt14} and the theory of liquid crystals~\cite{hkl86}. Its use in micromagnetics goes back to~\cite{melcher14}.
We extend his ideas and introduce a helicity calculus, which will prove useful in our analysis.
In Section~\ref{sec:strong}, we give an elementary proof for the uniqueness of strong solutions of LLG (Theorem~\ref{theorem:strong-strong}), which strongly relies on an energy argument for the dissipation of the micromagnetic bulk energy. In Section~\ref{sec:proof}, we show how to transfer the steps of the strong-strong uniqueness proof to obtain the weak-strong uniqueness result of Theorem~\ref{theorem:main}.

\bigskip

{\bf General notation.}
Throughout the paper, for subsets $X \subset \R^{d_1}$ and $Y \subset \R^{d_2}$, we denote by $\skp{\cdot}{\cdot}{X}$ the $L^2(X;Y)$ scalar product and by $\norm{\cdot}{X}$ the corresponding norm, where we assume that the set $Y$ is clear from the context.

%\clearpage
%!TEX root = weak-strong.tex

%%%%%%%%%%%%%%%%%%%%%%%%%%%%%%%%%%%%%%%%%%%%%%%%%%%%%%%%%%%%%%%%%%%%%%%%%%%%%%%%%%%
%%%%%%%%%%%%%%%%%%%%%%%%%%%%%%%%%%%%%%%%%%%%%%%%%%%%%%%%%%%%%%%%%%%%%%%%%%%%%%%%%%%
\section{Main Result}\label{sec:result}
%%%%%%%%%%%%%%%%%%%%%%%%%%%%%%%%%%%%%%%%%%%%%%%%%%%%%%%%%%%%%%%%%%%%%%%%%%%%%%%%%%%
%%%%%%%%%%%%%%%%%%%%%%%%%%%%%%%%%%%%%%%%%%%%%%%%%%%%%%%%%%%%%%%%%%%%%%%%%%%%%%%%%%%

%%%%%%%%%%%%%%%%%%%%%%%%%%%%%%%%%%%%%%%%%%%%%%%%%%%%%%%%%%%%%%%%%%%%%%%%%%%%%%%%%%%
\subsection{LLG equation}\label{subsec:llg}
%%%%%%%%%%%%%%%%%%%%%%%%%%%%%%%%%%%%%%%%%%%%%%%%%%%%%%%%%%%%%%%%%%%%%%%%%%%%%%%%%%%
In this paper, we consider the micromagnetic energy functional
\begin{equation}
\label{eq:result:llg-energy}
\begin{split}
	&\EE[\m, \vec{f}]
	:=
	\EE_\mathrm{ex}[\m] + \EE_\mathrm{DMI}[\m] + \EE_\mathrm{lo}[\m] +\EE_\mathrm{appl}[\m, \vec{f}] \\
	& \qquad :=
	\frac{\lex^2}{2} \int_\Omega |\nabla \m|^2 \d{x}
	+ \frac{\kappa}{2} \int_\Omega \m \cdot \curl \m \d{x}
	- \frac{1}{2} \int_\Omega \m \cdot \boldsymbol{\pi}(\m) \d{x}
	- \int_\Omega \m \cdot \vec{f} \d{x}.
\end{split}
\end{equation}
Here, $\EE_\mathrm{ex}$ is the exchange energy with exchange length $\lex > 0$, and $\EE_\mathrm{DMI}$ is the energy contribution associated with the Dzyaloshinskii--Moriya interaction with strength $\kappa \in \R$.
In $\EE_\mathrm{lo}$, we summarize lower-order contributions like anisotropy and stray field energy.
The operator $\boldsymbol{\pi}: L^2(\Omega;\R^3) \to L^2(\Omega;\R^3)$ is required to be linear, bounded, and self-adjoint, and to preserve certain regularity in the sense that
\begin{equation}\label{eq:pi-regularity}
	\norm{\pi(\u)}{L^\infty(0,T; C^1(\overline{\Omega}))}
	<
	\infty
	\quad \text{for all } \u \in C^3(\overline{\Omega} \times [0,T]; \R^3).
\end{equation}
Finally, $\EE_\mathrm{appl}$ describes the energy contribution due to a given time-dependent external field $\vec{f} \in C^1(\overline{\Omega} \times [0,\infty); \R^3)$.

Direct computation provides the functional derivative of the energy functional,
\begin{equation}
\label{eq:result:energy-derivative}
	\partial_{\m} \EE [\m, \vec{f}](\vec{\varphi})
	=
	\int_\Omega \Big( \lex \nabla \m : \nabla \vec{\varphi}
	+ \frac{\kappa}{2} \m \cdot \curl \vec{\varphi} + \frac{\kappa}{2} \vec{\varphi} \cdot \curl \m
	- \boldsymbol{\pi}(\m) \cdot \vec{\varphi}
	- \vec{f} \cdot \vec{\varphi} \Big)
	\d{x}.
\end{equation}
For $\m$ smooth enough and satisfying suitable boundary conditions (cf.\ \eqref{eq:result:strong-bc} below), integration by parts yields the effective field
\begin{equation*}
	\heff(\m)
	:=
	- \partial_{\m} \EE [\m, \vec{f}]
	=
	\lex^2 \Delta \m - \kappa \curl \m + \boldsymbol{\pi}(\m) + \vec{f}.
\end{equation*}
Together with natural boundary conditions of the energy \eqref{eq:result:llg-energy}, the strong form of LLG reads
\begin{subequations}
\label{eq:result:llg-strong-system}
	\begin{align}
	\label{eq:result:llg-strong}
		\partial_t \m
		=
		\alpha \m \times \partial_t \m - \m \times \heff(\m)
		& \quad \text{in } \Omega \times [0,T),\\
	\label{eq:result:strong-bc}
		\lex \partial_{\normalv} \m + \frac{\kappa}{\lex} \m \times \normalv = \vec{0}
		& \quad \text{on } \partial \Omega \times [0,T) , \\
		\m(0,\cdot) = \m_0
		& \quad \text{in } \Omega,
	\end{align}
\end{subequations}
where $\normalv$ is the outwards facing unit normal vector on $\partial \Omega$. 
%and $\m_0 \in H^1(\Omega; \S^2)$ satisfies the above boundary condition \eqref{eq:result:strong-bc} in order to allow for an integration by parts in space.
For solutions to \eqref{eq:result:llg-strong-system} the energy is conserved up to some \emph{dissipation terms}
\begin{equation}
\label{eq:def-disspative-terms}
	\DD[\m, \vec{f}](t)
	:=
	\int_{0}^{t} \!\!\! \alpha \norm{\partial_t \m}{\Omega}^2 \d{t}
	+ \int_{0}^{t} \!\!\! \skp{\partial_t \vec{f}}{\m}{\Omega} \d{t},
\end{equation}
i.e., there holds the \emph{energy equality}
\begin{equation}
\label{eq:strong-energy-equality}
	\EE[\m(t), \vec{f}(t)] + \DD[\m, \vec{f}](t)
	=
	\EE[\m_0, \vec{f}(0)]
	\quad \text{for all }
	t \in (0,T).
\end{equation}

\begin{remark}
	Formally, to derive the energy equality~\eqref{eq:strong-energy-equality}, one multiplies by $\alpha \partial_t \m - \heff(\m)$ equation \eqref{eq:result:llg-strong}, hence obtaining the conservation law
	\begin{equation}
	\label{eq:result:alternate-energy-inequality}
		\partial_t \m \cdot \big(\alpha \partial_t \m - \heff(\m)\big) = 0 ,
	\end{equation}
	pointwise in space and time.
	Then, \eqref{eq:strong-energy-equality} follows from integration over space and time and integration by parts in space.
	We note that \eqref{eq:result:alternate-energy-inequality} is stronger than \eqref{eq:strong-energy-equality} and will be exploited in the proof of Theorem~\ref{theorem:strong-strong}, which states uniqueness of strong solutions.
\end{remark}

Multiplying \eqref{eq:result:llg-strong} with an appropriate test function, we also obtain a weak form of LLG.
To restrict only to meaningful solutions in a physical sense, one has to incorporate a weak analogue to the energy equality~\eqref{eq:strong-energy-equality} into the definition of weak solutions.
To this end, we follow the lines of \cite{AS92}.

\begin{definition}\label{def:weak-solution}
	Given $\m_0 \in H^1(\Omega; \S^2)$, $\m \in L^\infty((0,\infty); H^1(\Omega; \S^2))$ is a global weak solution of \eqref{eq:result:llg-strong-system} if the following properties {\rm (i)--(iv)} are satisfied for almost all $T > 0$:
	\begin{enumerate}[label={\rm (\roman*)}]
		\item $\m \in H^1(\Omega \times (0,T); \S^2)$;
		
		\item $\m(0,\cdot) = \m_0$ in the sense of traces;
		
		\item for all $\vec{\varphi} \in  H^1(\Omega \times (0,T); \R^3)$ there holds
		\begin{equation}
		\label{eq:result:llg-weak}
			\int_{0}^{T} \skp{\partial_t \m}{\vec{\varphi}}{\Omega} \d{t}
			= 
			\int_{0}^{T} \alpha \skp{\partial_t \m}{\vec{\varphi} \times \m}{\Omega}
			+ \partial_{\m} \EE [\m, \vec{f}](\vec{\varphi} \times \m) \d{t};
		\end{equation}
		
		\item there holds the \emph{energy inequality}
		\begin{equation}
		\label{eq:result:energy-inequality}
			\EE[\m(T), \vec{f}(T)]
			+ \DD[\m, \vec{f}](T)
			\leq
			\EE[\m_0, \vec{f}(0)].
		\end{equation}
	\end{enumerate}
\end{definition}

%%%%%%%%%%%%%%%%%%%%%%%%%%%%%%%%%%%%%%%%%%%%%%%%%%%%%%%%%%%%%%%%%%%%%%%%%%%%%%%%%%%
\subsection{Main theorem}\label{subsec:theorem}
%%%%%%%%%%%%%%%%%%%%%%%%%%%%%%%%%%%%%%%%%%%%%%%%%%%%%%%%%%%%%%%%%%%%%%%%%%%%%%%%%%%
The main result of this paper is stated in the following theorem.

\begin{theorem}\label{theorem:main}
	Let $\m_0 \in H^1(\Omega; \S^2)$ and $T > 0$.
	Suppose that $\m_1 \in C^3(\overline{\Omega} \times [0,T])$ is a strong solution of \eqref{eq:result:llg-strong} and $\m_2$ is a global weak solution in the sense of Definition~\ref{def:weak-solution}.
	Then, it follows that
	\begin{equation*}
		\m_1 = \m_2
		\qquad
		\text{a.e.\ in } \Omega \times (0,T).
	\end{equation*}
\end{theorem}

\begin{remark}
%We use the fact that $\m_0$ satisfies the boundary conditions in equation~\eqref{eq:proof:estimate-helical} below.
%In \cite{DS14}, a similar computation is made without integration by parts in space (and hence avoidance of boundary conditions for $\m_0$). 
%However, the proof of~\cite{DS14} requires that $\partial_t \skp{\nabla \m_1}{\nabla \m_2}{\Omega}$ exists, which is unclear and avoided by the present argument.
%
The regularity assumptions for the strong solution are used, e.g., in~\eqref{eq:strong:expanded-estimate}, where the hidden $\nabla_\helical \Delta_\helical \m_1$ requires $C^3$ regularity in the interior. Moreover, the embedding theorems used in the proof of Lemma~\ref{lemma:Linfty-bounds} require $C^3$ regularity also up to the boundary, resulting in $\m_1 \in C^3(\overline{\Omega} \times [0,T])$.

Note that $\m_1 \in C^3(\overline{\Omega} \times [0,T])$ implies a higher regularity of $\m_0$ than assumed.
	In contrast to elliptic regularity theory, there is no known result that guarantees regularity of solutions based on smoothness of the initial condition alone.
	For the LLG equation, such results are only known with additional assumptions to $\m_0$ (see \cite{cf01a,cf01b,melcher12, ft17}) which, in our case, need not be fulfilled.
	Overall, weak-strong uniqueness is a non-trivial observation.
\end{remark}

\begin{remark}
	For the ease of presentation, we restrict to linear and self-adjoint lower-order terms $\boldsymbol{\pi}(\cdot)$.
	However, with slight modifications, general lower-order terms with pointwise nonlinearities can also be included if they are Lipschitz continuous and satisfy~\eqref{eq:pi-regularity}.
	This covers, for instance, usual anisotropy contributions, where the anisotropy density is a (pointwise) polynomial of the magnetization $\m$.
\end{remark}

%\clearpage
%!TEX root = weak-strong.tex

%%%%%%%%%%%%%%%%%%%%%%%%%%%%%%%%%%%%%%%%%%%%%%%%%%%%%%%%%%%%%%%%%%%%%%%%%%%%%%%%%%%
%%%%%%%%%%%%%%%%%%%%%%%%%%%%%%%%%%%%%%%%%%%%%%%%%%%%%%%%%%%%%%%%%%%%%%%%%%%%%%%%%%%
\section{Preliminaries}\label{sec:auxiliary}
%%%%%%%%%%%%%%%%%%%%%%%%%%%%%%%%%%%%%%%%%%%%%%%%%%%%%%%%%%%%%%%%%%%%%%%%%%%%%%%%%%%
%%%%%%%%%%%%%%%%%%%%%%%%%%%%%%%%%%%%%%%%%%%%%%%%%%%%%%%%%%%%%%%%%%%%%%%%%%%%%%%%%%%

We introduce some auxiliary results which will be used in the following sections.
First, we recall the so-called \emph{helical derivative} from \cite{melcher14}.
While it was originally only used to summarize exchange and Dzyaloshinskii--Moriya energy in a positive energy term (Lemma~\ref{lemma:helical-representation}), we introduce here a \emph{helicity calculus}, which makes our proofs clearer.

%%%%%%%%%%%%%%%%%%%%%%%%%%%%%%%%%%%%%%%%%%%%%%%%%%%%%%%%%%%%%%%%%%%%%%%%%%%%%%%%%%%
\subsection{Helical derivative}\label{subsec:helical}
%%%%%%%%%%%%%%%%%%%%%%%%%%%%%%%%%%%%%%%%%%%%%%%%%%%%%%%%%%%%%%%%%%%%%%%%%%%%%%%%%%%
For $i = 1,2,3$ and $\u \in H^1(\Omega; \R^3)$, define the \emph{partial helical derivative}, the \emph{helical gradient}, and the \emph{helical Laplacian} as
\begin{equation}
\label{eq:def-helical-derivative}
	\partial^\helical_i \u
	:=
	\lex \partial_i \u + \frac{\kappa}{\lex} (\u \! \times \! \vec{e}_i),
	\quad
	\nabla_\helical \u
	:=
	(\partial^\helical_1 \u, \partial^\helical_2 \u, \partial^\helical_3 \u),
	\quad
	\Delta_\helical \u
	:=
	\sum_{i=1}^{3} \partial^\helical_i \partial^\helical_i \u,
\end{equation}
respectively, where $\vec{e}_i \in \R^3$ is the $i$-th coordinate vector.
In the following, for a matrix $\vec{M} \in \R^{3 \times 3}$ and a (column) vector $\u \in \R^3$, the cross products $\vec{M} \times \u \in \R^{3 \times 3}$ and $\u \times \vec{M} \in \R^{3 \times 3}$ are understood to act column-wise.
The following lemma collects some rules for the helical derivative.

\begin{lemma}\label{lemma:helical-properties}
	The partial helical derivatives $\partial_i^\helical$ are linear operators $H^1(\Omega;\R^3) \to L^2(\Omega;\R^3)$.
	Furthermore, the helical gradient has the following properties {\rm (i)}--{\rm (iii)}.
	\begin{enumerate}[label={\rm (\roman*)}]
		\item For $\u_1 \in H^1(\Omega;\R^3)$ and $\u_2 \in C^2(\overline{\Omega};\R^3)$,
		there holds the integration by parts formula
		\begin{equation}
		\label{eq:auxiliary:helical-ibp}
		\skp{\nabla_\helical \u_1}{\nabla_\helical \u_2}{\Omega}
		=
		\skp{\u_1}{\nabla_\helical \u_2 \cdot \normalv}{\partial \Omega}
		- \skp{\u_1}{\Delta_\helical \u_2}{\Omega},
		\end{equation}
		where $\normalv$ is the outwards facing unit normal vector on $\partial \Omega$.
		\item For $\u \in C^2(\Omega \times (0,T); \R^3)$, it holds that
		\begin{equation*}
			\partial_t \nabla_\helical \u
			=
			\nabla_\helical \partial_t \u.
		\end{equation*}
		
		\item For $\u_1, \u_2 \in H^1(\Omega;\R^3)$,
		there holds the Leibniz rule
		\begin{equation}
		\label{eq:auxiliary:chainrule}
			\nabla_\helical (\u_1 \times \u_2)
			=
			\nabla_\helical \u_1 \times \u_2 + \u_1 \times \nabla_\helical \u_2.
		\end{equation}
	\end{enumerate}
\end{lemma}

\begin{proof}
	The claims follow from explicit computation and the corresponding identities for the partial derivatives $\partial_t$ and $\partial_i$ for $i=1,2,3$.
\end{proof}

The following lemma from~\cite{melcher14} states that, for magnetization fields, the energy of the helical derivative is the sum of exchange energy $\EE_\mathrm{ex}$ and Dzyaloshinskii--Moriya energy $\EE_\mathrm{DMI}$ plus some constant.
For the convenience of the reader, we give a short proof here.
\begin{lemma}\label{lemma:helical-representation}
	Let $\u \in H^1(\Omega; \S^2)$.
	Then, it holds that
	\begin{equation}
	\label{eq:auxiliary:helical-representation}
		\frac{1}{2} \int_\Omega  \abs{\nabla_\helical \u}^2 \d{x}
		=
		\int_\Omega \bigg(
		\frac{\lex^2}{2} \abs{\nabla \u}^2 + \kappa \, \u \cdot \curl \u + \frac{\kappa^2}{\lex^2}
		\bigg) \d{x}.
	\end{equation}
\end{lemma}

\begin{proof}
	Expanding the left-hand side, we see that
	\begin{align*}
		\int_\Omega \abs{\nabla_\helical \u}^2 \d{x}
		&=
		\sum_{i=1}^3 \int_\Omega \abs{\partial^\helical_i \u}^2 \d{x}\\
		&=
		\sum_{i=1}^3 \int_\Omega 
		\Big( \lex^2 \abs{\partial_i \u}^2 + 2 \kappa \, \partial_i \u \cdot (\u \times \vec{e}_i) + \frac{\kappa^2}{\lex^2} \abs{\u \times \vec{e}_i}^2 \Big) \d{x}.
	\end{align*}
	The first term on the right-hand side clearly yields $\lex^2 \abs{\nabla \u}^2$.
	For the second term, we use the identity $\curl \u = \sum_{i=1}^3 \vec{e}_i \times \partial_i \u$.
	Using the properties of the triple product, we infer that
	\begin{equation*}
		\sum_{i=1}^3 \partial_i \u \cdot (\u \times \vec{e}_i)
		=
		\u \cdot \curl \u.
	\end{equation*}
	For the last term, we get with $\abs{\u} = 1$ a.e. in $\Omega$ that
	\begin{equation*}
		\sum_{i=1}^3 \abs{\u \times \vec{e}_i}^2
		= 
		\sum_{i=1}^3 \big( \abs{\u}^2 \abs{\vec{e}_i}^2 - \abs{\u \cdot \vec{e}_i}^2 \big)
		=
		\sum_{i=1}^3 (1-\u_i^2) = 3-1 = 2,
	\end{equation*}
	where the first equality follows from the Lagrange identity
	\begin{equation*}
		(\vec{a} \times \vec{b}) \cdot (\vec{c} \times \vec{d})
		=
		(\vec{a} \cdot \vec{c}) (\vec{b} \cdot \vec{d})
		- (\vec{a} \cdot \vec{d}) (\vec{b} \cdot \vec{c}).
	\end{equation*}
	Finally, combining the above results shows the assertion.
\end{proof}

%%%%%%%%%%%%%%%%%%%%%%%%%%%%%%%%%%%%%%%%%%%%%%%%%%%%%%%%%%%%%%%%%%%%%%%%%%%%%%%%%%%
\subsection{Reformulation of LLG}\label{subsec:helicalLLG}
%%%%%%%%%%%%%%%%%%%%%%%%%%%%%%%%%%%%%%%%%%%%%%%%%%%%%%%%%%%%%%%%%%%%%%%%%%%%%%%%%%%
With the aid of Lemma~\ref{lemma:helical-representation}, we can define a shifted energy functional
\begin{equation}
\label{eq:def-helical-energy}
	\EE_\helical[\m, \vec{f}]
	:=
	\frac{1}{2} \int_\Omega |\nabla_\helical \m|^2 \d{x}
	- \frac{1}{2} \int_\Omega \m \cdot \boldsymbol{\pi}(\m) \d{x}
	- \int_\Omega \m \cdot \vec{f} \d{x}
	\stackrel{\eqref{eq:auxiliary:helical-representation}}{=}
	\EE[\m, \vec{f}] + \frac{\kappa^2}{\lex^2} |\Omega|.
\end{equation}
From the constant shift of $\EE$ with respect to $\EE_\helical$, we infer that their functional derivatives are the same.
Together with the boundary condition~\eqref{eq:result:strong-bc} (see also~\eqref{eq:helical-boundary-condition} below), the integration by parts rule~\eqref{eq:auxiliary:helical-ibp} allows to determine the functional derivative of the energy contribution $\frac{1}{2} \int_\Omega \abs{\nabla_\helical \m}^2 \d{x}$, which is $-\Delta_\helical \m$.
Thus, the strong form of LLG~\eqref{eq:result:llg-strong-system} can equivalently be reformulated in terms of the helical derivative:

\begin{subequations}
\begin{align}
\label{eq:llg-strong-helical}
	\partial_t \m
	&=
	\alpha \m \times \partial_t \m
	- \m \times \big( \Delta_\helical \m + \boldsymbol{\pi}(\m) + \vec{f} \big),
	&& \text{in} ~ \Omega \times [0,T),\\
\label{eq:helical-boundary-condition}
	\nabla_\helical \m \cdot \normalv
	&=
	\vec{0},
	&& \text{on} ~ \partial \Omega \times [0,T),\\
\label{eq:helical-initial-condition}
	\m(0, \cdot)
	&=
	\m_0,
	&& \text{in} ~ \Omega.
\end{align}
\end{subequations}
Since $\EE_\helical$ and $\EE$ differ only by a constant, the weak formulation of the previous equation coincides with the weak formulation \eqref{eq:result:llg-weak} and reads
\begin{equation}
\label{eq:llg-weak-helical}
	\int_{0}^{T} \skp{\partial_t \m}{\vec{\varphi}}{\Omega} \d{t}
	= 
	\int_{0}^{T} \alpha \skp{\partial_t \m}{\vec{\varphi} \times \m}{\Omega}
	+ \partial_{\m} \EE_\helical [\m, \vec{f}](\vec{\varphi} \times \m) \d{t}
\end{equation}
for all $\vec{\varphi} \in H^1(\Omega \times (0,T); \R^3)$.
The energy inequality~\eqref{eq:result:energy-inequality} for $\EE_\helical$ reads
\begin{equation}
\label{eq:energy-inequality-helical}
	\EE_\helical[\m(T), \vec{f}(T)]
	+ \DD[\m, \vec{f}](T)
	\leq
	\EE_\helical[\m_0, \vec{f}(0)]
	\quad \text{for almost all }
	T > 0.
\end{equation}

For the sake of conciseness, we introduce the notation
\begin{equation}
\begin{split}
\label{eq:psi-operators}
	\Psi[\u] &:= \alpha \partial_t \u - \Delta_\helical \u - \boldsymbol{\pi}(\u), \\
	\psi[\u_1, \u_2]
	&:=
	\alpha \skp{\partial_t \u_1}{\u_2}{\Omega}
	+ \skp{\nabla_\helical \u_1}{\nabla_\helical \u_2}{\Omega}
	- \skp{\boldsymbol{\pi}(\u_1)}{\u_2}{\Omega}.
\end{split}
\end{equation}
With this, we can rewrite the strong form~\eqref{eq:llg-strong-helical} and the weak form~\eqref{eq:llg-weak-helical} as
\begin{align}
\label{eq:strong-reformulation}
	\partial_t \m
	&=
	\m \times  \big( \Psi[\m] -  \vec{f} \big), \\
\label{eq:weak-reformulation}
	\int_{0}^{T} \skp{\partial_t \m}{\vec{\varphi}}{\Omega} \d{t}
	&= 
	\int_{0}^{T} \psi[\m, \vec{\varphi} \times \m]
	- \skp{\vec{f}}{\vec{\varphi} \times \m}{\Omega} \d{t},
\end{align}
respectively.
For functions $\u_1$ that are smooth enough such that $\Psi[\u_1]$ is defined, the integration by parts formula~\eqref{eq:auxiliary:helical-ibp} shows that
\begin{equation}
\label{eq:proof:Kcorrespondence}
	\psi[\u_1, \u_2]
	=
	\skp{\Psi[\u_1]}{\u_2}{\Omega}
	+ \skp{\nabla_\helical \u_1 \cdot \normalv}{\u_2}{\partial \Omega}.
\end{equation}
For strong solutions $\m$ of LLG, it follows that $\psi[\m, \vec{\varphi} \times \m] = \skp{\Psi[\m]}{\vec{\varphi} \times \m}{\Omega}$.

%%%%%%%%%%%%%%%%%%%%%%%%%%%%%%%%%%%%%%%%%%%%%%%%%%%%%%%%%%%%%%%%%%%%%%%%%%%%%%%%%%%
\subsection{Smoothness of lower-order terms}\label{subsec:lo-terms}
%%%%%%%%%%%%%%%%%%%%%%%%%%%%%%%%%%%%%%%%%%%%%%%%%%%%%%%%%%%%%%%%%%%%%%%%%%%%%%%%%%%
In the proof of Theorem~\ref{theorem:main}, we need to bound the linear lower-order terms $\boldsymbol{\pi}(\cdot)$ in certain norms.
This will be done with the help of the assumptions made in Section~\ref{subsec:llg}, in particular, $L^2$ continuity and~\eqref{eq:pi-regularity}.
We note that these assumptions are not too restrictive in the sense that (at least) they are satisfied by the most relevant lower-order terms, namely uniaxial anisotropy and stray field:
The uniaxial anisotropy density reads $\varphi(\m) = 1 - (\m \cdot \vec{e})^2$ with $\vec{e} \in \S^2$ being the so-called easy axis.
Hence, its contribution to $\boldsymbol{\pi}(\cdot)$ reads $\boldsymbol{\pi}_\textrm{aniso}(\m) = 2(\m \cdot \vec{e}) \vec{e}$ and satisfies all assumptions made.
As far as the stray field is concerned, $L^2$ continuity is well known in the literature \cite{praetorius04, fmrs19}, while the validity of~\eqref{eq:pi-regularity} is less obvious but shown in the following lemma.

\begin{lemma}\label{lemma:Linfty-bounds}
	Let $\u \in C^3(\overline{\Omega} \times [0,T]; \R^3)$ and $\theta \in (0,1)$.
	Let further $\boldsymbol{h}_s$ be the stray field contribution to the lower order terms $\boldsymbol{\pi}$.
	Then, $\boldsymbol{h}_s(\u) \in C^{2,\theta}(\overline{\Omega} \times [0,T]; \R^3)$ and
	\begin{equation}
	\label{eq:pi-bounded}
		\norm{\boldsymbol{h}_s(\u)}{C^{2, \theta}(\overline{\Omega} \times [0,T]; \R^3)}
		\leq
		C \norm{\u}{C^3(\overline{\Omega} \times [0,T]; \R^3)},
	\end{equation}
	where $C > 0$ depends only on $\theta$ and $\Omega$.
\end{lemma}

\begin{proof}
	Since $\u \in C^3(\overline{\Omega} \times [0,T]; \R^3) \subset W^{3,p}(\Omega \times (0,T); \R^3)$ for all $1 \leq p \leq \infty$, we infer from~\cite[Proposition~3.1]{cfg07} that $\boldsymbol{h}_s(\u) \in W^{3,p}(\Omega \times (0,T); \R^3)$ for all $1 < p < \infty$.
	Due to $\Omega$ being bounded and Lipshitz, Morrey's embedding \cite[Theorem~12.55]{leoni17} yields for $p > 3$ the continuous embedding
	\begin{equation*}
		W^{3,p}(\Omega \times (0,T); \R^3)
		\subset
		C^{2,\theta}(\overline{\Omega} \times [0,T]; \R^3)
		\quad \text{with} \quad
		\theta = 1-3/p > 0.
	\end{equation*}
	Since $p > 3$ was arbitrary, this proves the claim.
\end{proof}

%%%%%%%%%%%%%%%%%%%%%%%%%%%%%%%%%%%%%%%%%%%%%%%%%%%%%%%%%%%%%%%%%%%%%%%%%%%%%%%%%%%
\subsection{Inequalities}\label{subsec:inequalities}
%%%%%%%%%%%%%%%%%%%%%%%%%%%%%%%%%%%%%%%%%%%%%%%%%%%%%%%%%%%%%%%%%%%%%%%%%%%%%%%%%%%
For later reference, we collect here the following two well-known inequalities. 
First, we cite an appropriate version of the Gronwall inequality from \cite[Appendix~B.2.k]{evans}.

\begin{lemma}[Gronwall inequality]\label{prop:gronwall}
	Let $u : [0,T] \to \R$ be nonnegative and integrable such that
	\begin{equation*}
		u(t)
		\leq
		C \int_0^t u(\tau) \d{\tau}
		\quad \text{for all} \quad
		0 \leq t \leq T
	\end{equation*}
	for a constant $C \geq 0$. Then, it follows that
	\begin{equation*}
		u(t) = 0
		\quad \text{for almost all} \quad
		0 \leq t \leq T. \hfill \qquad \qed
	\end{equation*}
\end{lemma}

Next, we state the Poincar\'e inequality in time.
\begin{lemma}[Poincar\'e inequality]\label{prop:poincare}
	Let $\u \in H^1((0,T); L^2(\Omega; \R^d))$ with $\w(0) = 0$.
	Then, for all $t \in [0,T]$, it holds that
	\begin{equation}
	\label{eq:poincare}
		\int_0^t \norm[]{\u(\tau)}{\Omega}^2 \d{\tau}
		\leq
		t^2 \int_0^t \norm[]{\partial_t \u(\tau) }{\Omega}^2 \d{\tau}. \qquad \qed
	\end{equation}
\end{lemma}

%\begin{proof}
%	By the fundamental theorem of calculus, for every $0 \leq \tau \leq t$ it holds that
%	\begin{align*}
%		\norm{ \u(\tau) }{\Omega}
%		&=
%		\norm[\Big]{\int_0^\tau \partial_t \u(s) \d{s}}{\Omega}
%		\leq
%		\int_0^\tau \norm{ \partial_t \u(s) }{\Omega} \d{s} \\
%		&\qquad \leq
%		\tau^{1/2} \Big( \int_0^\tau \norm{ \partial_t \u(s) }{\Omega}^2 \d{s} \Big)^{1/2}
%		\leq
%		t^{1/2} \Big( \int_0^t \norm{ \partial_t \u(s) }{\Omega}^2 \d{s} \Big)^{1/2}.
%	\end{align*}
%	From this, we infer that
%	\begin{equation*}
%		\int_0^t \norm{ \u(\tau) }{\Omega}^2 \d{\tau}
%		\leq
%		\int_0^t t \int_0^t \norm{ \partial_t \u(s) }{\Omega}^2 \d{s} \d{\tau}
%		=
%		t^2 \int_0^t \norm{ \partial_t \u(\tau) }{\Omega}^2 \d{\tau}.
%	\end{equation*}
%	This concludes the proof.
%\end{proof}

%\clearpage
%!TEX root = weak-strong.tex

%%%%%%%%%%%%%%%%%%%%%%%%%%%%%%%%%%%%%%%%%%%%%%%%%%%%%%%%%%%%%%%%%%%%%%%%%%%%%%%%%%%
%%%%%%%%%%%%%%%%%%%%%%%%%%%%%%%%%%%%%%%%%%%%%%%%%%%%%%%%%%%%%%%%%%%%%%%%%%%%%%%%%%%
\section{Uniqueness of strong solutions}\label{sec:strong}
%%%%%%%%%%%%%%%%%%%%%%%%%%%%%%%%%%%%%%%%%%%%%%%%%%%%%%%%%%%%%%%%%%%%%%%%%%%%%%%%%%%
%%%%%%%%%%%%%%%%%%%%%%%%%%%%%%%%%%%%%%%%%%%%%%%%%%%%%%%%%%%%%%%%%%%%%%%%%%%%%%%%%%%

In this section, we prove uniqueness of strong solutions of \eqref{eq:result:llg-strong-system}.
To illustrate the idea of the proof of the strong-weak uniqueness of solutions of LLG (Theorem~\ref{theorem:main}), we give a very simple argument which will also serve as an overture.
A similar idea was already used by~\cite{cimrak07}.

\begin{theorem}\label{theorem:strong-strong}
	Let $\m_0 \in H^1(\Omega; \S^2)$ and $T > 0$.
	Suppose $\m_1, \m_2 \in C^3(\overline{\Omega} \times [0,T]; \S^2)$ are strong solutions of \eqref{eq:result:llg-strong-system}.
	Then, it follows that
	\begin{equation*}
	\m_1 = \m_2
	\qquad
	\text{on } \Omega \times (0,T).
	\end{equation*}
\end{theorem}

We begin by investigating the value of the conservation law~\eqref{eq:result:alternate-energy-inequality} for strong solutions at the difference of the two solutions.
\begin{lemma}\label{lemma:strong-estimate}
	Suppose the assumptions of Theorem~\ref{theorem:strong-strong}.
	For all times $0< t < T$, the difference $\w := \m_2 - \m_1$ then satisfies that
	\begin{equation}
	\label{eq:strong:estimate}
		\skp[\big]{\partial_t \w}{\Psi[\w]}{\Omega}
		=
		\skp[\big]{\Psi[\w]}{\w \times (\Psi[\m_1] - \vec{f})}{\Omega}.
	\end{equation}
\end{lemma}

\begin{proof}
	Note that $\Psi$ is linear, whence we have $\Psi[\m_2] - \Psi[\m_1] = \Psi[\w].$
	The alternate form of the energy preservation~\eqref{eq:result:alternate-energy-inequality} reads $\skp{\partial_t \m_j}{\Psi[\m_j] - \vec{f}}{} = 0$.
	Hence, we see that
	\begin{align}
	\nonumber
		\skp[]{\partial_t \w}{\Psi[\w]}{\Omega}
		&=
		\skp[]{\partial_t \m_1}{\Psi[\m_1]}{\Omega}
		+ \skp[]{\partial_t \m_2}{\Psi[\m_2]}{\Omega}
		- \skp[]{\partial_t \m_1}{\Psi[\m_2]}{\Omega}
		- \skp[]{\partial_t \m_2}{\Psi[\m_1]}{\Omega}\\
	\label{eq:strong:remainder}
		&\stackrel{\mathclap{\eqref{eq:result:alternate-energy-inequality}}}{=}
		\skp{\partial_t \m_1}{\vec{f}}{\Omega}
		+ \skp{\partial_t \m_2}{\vec{f}}{\Omega}
		- \skp[]{\partial_t \m_1}{\Psi[\m_2]}{\Omega}
		- \skp[]{\partial_t \m_2}{\Psi[\m_1]}{\Omega}\\
	\nonumber
		&=
		- \skp[]{\partial_t \m_1}{\Psi[\m_2] - \vec{f}}{\Omega}
		- \skp[]{\partial_t \m_2}{\Psi[\m_1] - \vec{f}}{\Omega}.
	\end{align}
	Using the strong form~\eqref{eq:strong-reformulation} for $\partial_t \m_1$ and $\partial_t \m_2$, together with basic properties of the triple product, we get
	\begin{align*}
		\skp[]{\partial_t \w}{\Psi[\w]}{\Omega}
		&\stackrel{\mathclap{\eqref{eq:strong-reformulation}}}{=}
		- \skp[]{\m_1 \times (\Psi[\m_1] - \vec{f})}{\Psi[\m_2] - \vec{f}}{\Omega}
		- \skp[]{\m_2 \times (\Psi[\m_2] - \vec{f})}{\Psi[\m_1] - \vec{f}}{\Omega}\\
		&=
		- \skp[]{\m_1 \times (\Psi[\m_1] - \vec{f})}{\Psi[\m_2] - \vec{f}}{\Omega}
		+ \skp[]{\m_2 \times (\Psi[\m_1] - \vec{f})}{\Psi[\m_2] - \vec{f}}{\Omega}\\
		&=
		\skp[]{\w \times (\Psi[\m_1] - \vec{f})}{\Psi[\m_2] - \vec{f}}{\Omega}
		= \skp[]{\w \times (\Psi[\m_1] - \vec{f})}{\Psi[\w]}{\Omega}.
	\end{align*}
	This proves the assertion.
\end{proof}

To prove Theorem~\ref{theorem:strong-strong}, we apply the Gronwall lemma (Lemma~\ref{prop:gronwall}) to the estimate~\eqref{eq:strong:estimate} of Lemma~\ref{lemma:strong-estimate}.

\begin{proof}[Proof of Theorem~\ref{theorem:strong-strong}]
	We expand the estimate \eqref{eq:strong:estimate} to obtain that
	\begin{align*}
		& \skp[\big]{\partial_t \w}{\alpha \partial_t \w - \Delta_\helical \w - \boldsymbol{\pi}(\w)}{\Omega}
		=
		\skp[\big]{\partial_t \w}{\Psi[\w]}{\Omega} \\
		& \qquad \qquad \stackrel{\eqref{eq:strong:estimate}}{=}
		\skp[\big]{\Psi[\w]}{\w \times (\Psi[\m_1] - \vec{f})}{\Omega}
		=
		\skp[\big]{\alpha \partial_t \w - \Delta_\helical \w - \boldsymbol{\pi}(\w)}{\w \times (\Psi[\m_1] - \vec{f})}{\Omega}.
	\end{align*}
	Using the properties of the helical derivative from Lemma~\ref{lemma:helical-properties} and the fact that $\w$ satisfies the boundary condition~\eqref{eq:helical-boundary-condition}, we see that
	\begin{equation}
	\label{eq:strong:expanded-estimate}
	\begin{split}
		&\alpha \norm[]{\partial_t \w}{\Omega}^2
		+ \frac{1}{2} \partial_t \norm[]{\nabla_\helical \w}{\Omega}^2
		- \skp[\big]{\partial_t \w}{\boldsymbol{\pi}(\w)}{\Omega}
		= 
		\alpha \skp[\big]{\partial_t \w}{\w \times (\Psi[\m_1] - \vec{f})}{\Omega} \\
		& \hspace{90pt} + \skp[\big]{\nabla_\helical \w}{\w \times \nabla_\helical (\Psi[\m_1] - \vec{f})}{\Omega} 
		- \skp[\big]{\boldsymbol{\pi}(\w)}{\w \times (\Psi[\m_1] - \vec{f})}{\Omega}
		 .
	\end{split}
	\end{equation}
	Recall from Section~\ref{subsec:llg} that $\boldsymbol{\pi}$ is $L^2$-bounded and satisfies~\eqref{eq:pi-regularity}.
	Because of the smoothness of $\m_1 \in C^3(\overline{\Omega} \times [0,T]; \R^3)$ and $\vec{f} \in C^1(\overline{\Omega} \times [0,T]; \R^3)$, this implies that
	\begin{align*}
		\norm{\boldsymbol{\pi}(\m_1)}{\Omega}
		\leq
		C_\pi \norm{\m_1}{\Omega}
		\quad \text{and} \quad
		\norm{\Psi[\m_1] - \vec{f}}{L^\infty(\Omega)}
		+ \norm{\nabla (\Psi[\m_1] - \vec{f})}{L^\infty(\Omega)}
		\leq
		C_\Psi
	\end{align*}
	with $C_\pi > 0$ depending only on $\Omega$ and the problem parameters, and $C_\Psi > 0$ additionally depending on $\norm{\m_1}{C^3(\overline{\Omega} \times [0,T]; \R^3)}$ and $\norm{\vec{f}}{C^1(\overline{\Omega} \times [0,T]; \R^3)}$.
	Moreover, recall the Young inequality
	\begin{equation*}
	\label{eq:young}
		ab \leq \frac{\delta}{2} a^2 + \frac{1}{2\delta} b^2
		\quad \text{for all} ~
		a, b \geq 0 \text{ and } \delta > 0.
	\end{equation*}
	Together with the Cauchy--Schwarz inequality, we bound the scalar product terms in~\eqref{eq:strong:expanded-estimate} as follows:
	\begin{align*}
		\skp[\big]{\partial_t \w}{\boldsymbol{\pi}(\w)}{\Omega}
		&\leq
		\frac{\delta}{2} \norm[]{ \partial_t \w }{\Omega}^2 + \frac{1}{2 \delta} C_{\pi}^2 \, \norm[]{ \w }{\Omega}^2, \\
		\alpha \skp[\big]{\partial_t \w}{\w \times (\Psi[\m_1] - \vec{f})}{\Omega}
		&\leq
		\frac{\alpha \delta}{2} \norm[]{ \partial_t \w }{\Omega}^2 + \frac{\alpha}{2 \delta} C_\Psi^2 \norm[]{ \w }{\Omega}^2, \\
		\skp[\big]{\nabla_\helical \w}{\w \times \nabla_\helical (\Psi[\m_1] - \vec{f})}{\Omega}
		&\leq
		\frac{1}{2} \norm[]{ \nabla_\helical \w }{\Omega}^2 + \frac{1}{2} C_\Psi^2 \norm[]{ \w }{\Omega}^2, \\
		- \skp[\big]{\boldsymbol{\pi}(\w)}{\w \times (\Psi[\m_1] - \vec{f})}{\Omega}
		&\leq
		\frac{1}{2} C_{\pi}^2 \, \norm[]{ \w }{\Omega}^2 + \frac{1}{2} C_\Psi^2 \norm[]{ \w }{\Omega}^2.
	\end{align*}
	
	Using these bounds in \eqref{eq:strong:expanded-estimate} and summing similar terms, we obtain that
	\begin{equation}
	\label{eq:summary}
		C_\ell(\delta) \norm[]{ \partial_t \w }{\Omega}^2
		+ \frac{1}{2} \partial_t \norm[]{ \nabla_\helical \w}{\Omega}^2
		\leq
		\frac{1}{2} \norm[]{ \nabla_\helical \w }{\Omega}^2
		+ C_r(\delta) \norm[]{ \w }{\Omega}^2,
	\end{equation}
	where
	\begin{equation*}
		C_\ell(\delta)
		:=
		\alpha - \frac{\alpha \delta + \delta}{2}
		\quad \text{and} \quad
		C_r(\delta)
		:=
		\Big( \frac{\alpha}{2 \delta} + 1 \Big) C_\Psi^2 + \Big( \frac{1}{2} + \frac{1}{2 \delta} \Big) C_{\pi}^2.
	\end{equation*}
	Integrating \eqref{eq:summary} over $(0,t)$, where $0 < t < T$, and observing that
	\begin{equation}
	\nabla_\helical \w(0) = \nabla_\helical(\m_2(0) - \m_1(0)) = 0,
	\end{equation} 
	we thus arrive at
	\begin{equation*}
		C_\ell(\delta)
		\int_0^t \norm[]{ \partial_t \w }{\Omega}^2 \d{t}
		+ \frac{1}{2}  \norm[]{ \nabla_\helical \w (t) }{\Omega}^2
		\leq
		\frac{1}{2} \int_0^t \norm[]{ \nabla_\helical \w }{\Omega}^2 \d{t}
		+ C_r(\delta) \int_0^t \norm[]{\w}{\Omega}^2 \d{t}.
	\end{equation*}
	Applying the Poincar\'e inequality~\eqref{eq:poincare} to the last term, we finally get that	
	\begin{equation}
	\label{eq:strong:pre-gronwall-estimate}
		\big( C_\ell(\delta) - C_r(\delta) t^2 \big)
		\int_0^t \norm[]{ \partial_t \w }{\Omega}^2 \d{t}
		+ \frac{1}{2}  \norm[]{ \nabla_\helical \w (t) }{\Omega}^2
		\leq
		\frac{1}{2} \int_0^t \norm[]{ \nabla_\helical \w }{\Omega}^2 \d{t}.
	\end{equation}
	Note that $C_r(\delta) > 0$ for all $\delta > 0$, while $C_\ell(\delta) > 0$ only for $\delta > 0$ being sufficiently small.
	Choose $\delta$ small enough to ensure that $C_\ell(\delta) > 0$.
	Choose $0 < t^2 < C_\ell(\delta) / C_r(\delta) =: T_\ast^2$ so that $C_\ell(\delta) - C_r(\delta) t^2 > 0$.
	Then, the first term on the left-hand side of~\eqref{eq:strong:pre-gronwall-estimate} is positive. Hence, we can ignore it and obtain that
	\begin{equation*}
		\norm[]{ \nabla_\helical \w (t) }{\Omega}^2
		\leq
		\int_0^t \norm[]{ \nabla_\helical \w }{\Omega}^2 \d{t}
		\qquad \text{for all }
		0 < t < \min \{T, T_\ast\}.
	\end{equation*}
	By applying Gronwall's lemma (Lemma~\ref{prop:gronwall}), we conclude that $\norm[]{ \nabla_\helical \w (t) }{\Omega}^2 = 0$ for all $0 < t < \min \{T, T_\ast\}$.
	Using this in \eqref{eq:strong:pre-gronwall-estimate}, we obtain that
	\begin{equation*}
		\big( C_\ell(\delta) - C_r(\delta) t^2 \big)
		\int_0^t \norm[]{ \partial_t \w }{\Omega}^2 \d{t}
		\leq
		0.
	\end{equation*}
	Since the prefactor is positive, we have $\int_0^t \alpha \norm[]{ \partial_t \w }{\Omega}^2 \d{t} = 0$ and hence $\m_1(t) = \m_2(t)$ for all $0 \leq t \leq \min \{T, T_\ast\}$.
	Since $T_\ast$ depends only on the norm of the smooth functions $\m_1$ and $\vec{f}$, we can repeat this argument on any interval of the same length $T_\ast$, as long as $\m_1$ and $\m_2$ are defined.
	This concludes the proof.
\end{proof}

\begin{remark}\label{rem:smoothness}
	Note that the identity \eqref{eq:strong:estimate} relies on the regularity of $\m_2$, since $\Psi[\m_2]$ contains the helical Laplacian.
	However, in the proof of Theorem~\ref{theorem:strong-strong} we use integration by parts to obtain~\eqref{eq:strong:expanded-estimate}, which does not assume more regularity of $\w$ (and thus of $\m_2$) than $H^1(\Omega \times (0,T); \R^3)$.
	Our strategy for weak solutions in Section~\ref{sec:proof} is to reproduce the estimate~\eqref{eq:strong:expanded-estimate} by following the steps in this section.
\end{remark}

%\clearpage
%!TEX root = weak-strong.tex

%%%%%%%%%%%%%%%%%%%%%%%%%%%%%%%%%%%%%%%%%%%%%%%%%%%%%%%%%%%%%%%%%%%%%%%%%%%%%%%%%%%
%%%%%%%%%%%%%%%%%%%%%%%%%%%%%%%%%%%%%%%%%%%%%%%%%%%%%%%%%%%%%%%%%%%%%%%%%%%%%%%%%%%
\section{Proof of Theorem~\ref{theorem:main}}\label{sec:proof}
%%%%%%%%%%%%%%%%%%%%%%%%%%%%%%%%%%%%%%%%%%%%%%%%%%%%%%%%%%%%%%%%%%%%%%%%%%%%%%%%%%%
%%%%%%%%%%%%%%%%%%%%%%%%%%%%%%%%%%%%%%%%%%%%%%%%%%%%%%%%%%%%%%%%%%%%%%%%%%%%%%%%%%%

We proceed as in the proof of Theorem~\ref{theorem:strong-strong} and prove an estimate analogous to Lemma~\ref{lemma:strong-estimate}.
To this end, we need a preliminary result which mimics the identity~\eqref{eq:strong:remainder}.
Note that in~\eqref{eq:strong:remainder} the source term $\vec{f}$ does not show up on the left-hand side of the identity, whereas in the following estimates it is present on the left-hand sides in the energy and dissipation term.
This, however, is compensated by additional $\vec{f}$-terms on the right-hand sides, which will cancel out in the end.
In particular, the later proof of Theorem~\ref{theorem:main} can follow the overall structure of the proof of Theorem~\ref{theorem:strong-strong}.

\begin{lemma}\label{lemma:proof:weak-estimate}
	Suppose the assumptions of Theorem~\ref{theorem:main}.
	For all $0 < t < T$, the difference $\w := \m_2 - \m_1$ then satisfies that
	\begin{equation}
	\label{eq:proof:estimate1}
		\EE_\helical[\w(t), \vec{f}(t)]  +  \DD[\w, \vec{f}](t)
		\leq
		- \!\! \int_{0}^{t} \!\!
		\Big( \! \psi[\m_2, \partial_t \m_1]
		+ \skp{\partial_t \m_2}{\Psi[\m_1]}{\Omega}
		- 2 \skp{\vec{f}}{\partial_t \m_1}{\Omega} \!
		\Big) \d{t}.
	\end{equation}
\end{lemma}

\begin{proof}
	We start by examining the helical energy of the difference $\w$:
	\begin{equation}
	\label{eq:proof:energy-difference}
	\begin{split}
		\EE_\helical[\w(t), \vec{f}(t)]
		& \stackrel{\eqref{eq:def-helical-energy}}{=}
		\EE_\helical[\m_1(t), \vec{f}(t)] + \EE_\helical[\m_2(t), \vec{f}(t)]
		+ 2\skp{\vec{f}(t)}{\m_1(t)}{\Omega} \\
		& \qquad
		- \skp{\nabla_\helical \m_1(t)}{\nabla_\helical \m_2(t)}{\Omega}
		+ \skp{\boldsymbol{\pi}(\m_1(t))}{\m_2(t)}{\Omega}.
	\end{split}
	\end{equation}
	Analogously, we can treat the dissipation terms of $\w$ to obtain that
	\begin{equation}
	\label{eq:proof:dissipative-difference}
		\DD[\w, \vec{f}](t)
		\stackrel{\eqref{eq:def-disspative-terms}}{=}
		\DD[\m_1, \vec{f}](t) + \DD[\m_2, \vec{f}](t)
		- 2 \int_{0}^{t} \Big( \skp{\partial_t \vec{f}}{\m_1}{\Omega} 
		+ \alpha \skp[]{\partial_t \m_1}{\partial_t \m_2}{} \Big) \d{t}.
	\end{equation}
	We estimate all terms of the right-hand sides of~\eqref{eq:proof:energy-difference}--\eqref{eq:proof:dissipative-difference}, starting with the energy and dissipation terms of $\m_1$ and $\m_2$.
	The energy inequality~\eqref{eq:energy-inequality-helical} yields that
	\begin{equation}
	\label{eq:proof:estimate-energy}
	\begin{split}
		&\EE_\helical[\m_1(t), \vec{f}(t)] + \DD[\m_1, \vec{f}](t)
		+ \EE_\helical[\m_2(t), \vec{f}(t)] + \DD[\m_2, \vec{f}](t) \\
		& \hspace{80pt} \stackrel{\eqref{eq:energy-inequality-helical}}{\leq}
		2 \, \EE_\helical[\m_0, \vec{f}(0)]
		\stackrel{\eqref{eq:def-helical-energy}}{=}
		\norm{\nabla_\helical \m_0}{\Omega}^2
		- \skp{\boldsymbol{\pi}(\m_0)}{\m_0}{\Omega}
		- 2 \skp{\vec{f}(0)}{\m_0}{\Omega}.
	\end{split}
	\end{equation}
	For the external field, we integrate by parts in time to obtain that
	\begin{equation}
	\label{eq:proof:estimate-external}
		2 \skp{\vec{f}(t)}{\m_1(t)}{\Omega}
		- 2 \int_{0}^{t} \skp{\partial_t \vec{f}}{\m_1}{\Omega} \d{t}
		= 2 \skp{\vec{f}(0)}{\m_0}{\Omega}
		+ 2 \int_{0}^{t} \skp{\vec{f}}{\partial_t \m_1}{\Omega} \d{t}.
	\end{equation}
	The terms involving the helical derivative can be integrated by parts in space and time.
	The boundary terms vanish since $\m_0$ and $\m_1$ satisfy the boundary condition of the strong form~\eqref{eq:helical-boundary-condition}.
	Hence, we obtain that
	\begin{align}
	\nonumber
		- \skp{\nabla_\helical \m_1(t)}{\nabla_\helical \m_2(t)}{\Omega}
		&\stackrel{\mathclap{\eqref{eq:auxiliary:helical-ibp}}}{=}
		\skp{\Delta_\helical \m_1(t)}{\m_2(t)}{\Omega}
		=
		\int_{0}^{t} \partial_t \skp{\Delta_\helical \m_1}{\m_2}{\Omega} \d{t}
		+ \skp{\Delta_\helical \m_0}{\m_0}{\Omega}\\
	\label{eq:proof:estimate-helical}
		&\stackrel{\mathclap{\eqref{eq:auxiliary:helical-ibp}}}{=}
		\int_{0}^{t}  \Big( - \skp{\partial_t \nabla_\helical \m_1}{\nabla_\helical \m_2}{\Omega} + \skp{\Delta_\helical \m_1}{\partial_t \m_2}{\Omega} \Big) \d{t}
		- \norm{\nabla_\helical \m_0}{\Omega}^2.
	\end{align}
	Note that integration by parts in space is necessary in order for the integration by parts in time to be well-defined.
	Applying integration by parts and using that $\boldsymbol{\pi}$ is self-adjoint, we get that
	\begin{equation}
	\label{eq:proof:estimate-lowerorder}
	\begin{split}
		&\skp{\boldsymbol{\pi}(\m_1(t))}{\m_2(t)}{\Omega}
		- \skp{\boldsymbol{\pi}(\m_0)}{\m_0}{\Omega} \\
		& \qquad = \int_{0}^{t} \Big( \skp{\boldsymbol{\pi}(\m_1(t))}{\partial_t \m_2(t)}{\Omega}
		+ \skp{\partial_t \m_1(t)}{\boldsymbol{\pi}(\m_2(t))}{\Omega} \Big) \d{t}.
	\end{split}
	\end{equation}
	
	Finally we sum the identities~\eqref{eq:proof:energy-difference}--\eqref{eq:proof:estimate-lowerorder}.
	After some cancellations we obtain that
	\begin{align*}
		\EE_\helical[\w(t), \vec{f}(t)]  +  \DD[\w, \vec{f}](t)
		\leq
		&- 2 \alpha \int_{0}^{t} \skp[]{\partial_t \m_1}{\partial_t \m_2}{} \d{t}
		+ 2 \int_{0}^{t} \skp{\vec{f}}{\partial_t \m_1}{\Omega} \d{t} \\
		&+ \int_{0}^{t}  \Big( - \skp{\partial_t \nabla_\helical \m_1}{\nabla_\helical \m_2}{\Omega}
			+ \skp{\Delta_\helical \m_1}{\partial_t \m_2}{\Omega} \Big) \d{t} \\
		&+ \int_{0}^{t} \Big( \skp{\boldsymbol{\pi}(\m_1(t))}{\partial_t \m_2(t)}{\Omega}
			+ \skp{\partial_t \m_1(t)}{\boldsymbol{\pi}(\m_2(t))}{\Omega} \Big) \d{t}.
	\end{align*}
	Combining the terms into $\psi$ and $\Psi$ according to~\eqref{eq:psi-operators}, we ultimately prove~\eqref{eq:proof:estimate1}.
\end{proof}

The result of the last lemma is the analogue of equation~\eqref{eq:strong:remainder} in the proof of Lemma~\ref{lemma:strong-estimate}.
There we commence to use the strong form~\eqref{eq:strong-reformulation} for both, $\m_1$ and $\m_2$.
However, this is not possible in \eqref{eq:proof:estimate1} due to the lack of regularity of $\m_2$.
For this reason, the first term in the right-hand side of~\eqref{eq:proof:estimate1} does \emph{not} read $\skp{\Psi[\m_2]}{\partial_t \m_1}{\Omega}$ and for the third, we cannot expand $\partial_t \m_2$ so easily.
Nevertheless, we recover a weak analogue of Lemma~\ref{lemma:strong-estimate} by an approximation argument.

\begin{lemma}
	Let the assumptions of Theorem~\ref{theorem:main} hold.
	For the difference $\w := \m_2 - \m_1$ and every $0 < t < T$, it holds that
	\begin{equation}
	\label{eq:proof:estimate2}
	\begin{split}
		\EE_\helical[\w(t), \vec{f}(t)] + \DD[\w, \vec{f}](t)
		&\leq
		\int_{0}^{t}
		\psi[\w, \w \times (\Psi[\m_1] - \vec{f})] \d{t} \\
		& \qquad \qquad - \skp{\vec{f}(t)}{\w(t)}{\Omega}
		+ \int_{0}^{t} \skp{\partial_t \vec{f}}{\w}{\Omega}
		\d{t}.
	\end{split}
	\end{equation}
\end{lemma}

\begin{proof}
	Since $\m_2 \in H^1(\Omega \times (0,T); \S^2)$, there exists a sequence $(\m_2^\varepsilon)_\varepsilon \in C^\infty(\overline{\Omega} \times [0,T]; \R^3)$ such that $\m_2^\varepsilon \to \m_2$ in $H^1(\Omega \times (0,T); \R^3)$ for $\varepsilon \to 0$.
	We further define $\w^\varepsilon := \m_2^\varepsilon - \m_1$ and note that also $\w^\varepsilon \to \w = \m_2 - \m_1$ in $H^1(\Omega \times (0,T); \R^3)$ for $\varepsilon \to 0$.
	
	The estimate~\eqref{eq:proof:estimate1} now reads
	\begin{equation}
	\label{eq:proof:limes-estimate}
	\begin{split}
		&\EE_\helical[\w(t), \vec{f}(t)] + \DD[\w, \vec{f}](t) \\
		&\hspace{50pt} \leq
		- \lim\limits_{\varepsilon \to 0}
		\int_{0}^{t} \!\! \Big(
		\psi[\m_2^\varepsilon, \partial_t \m_1]
		+ \skp{\partial_t \m_2^\varepsilon}{\Psi[\m_1]}{\Omega}
		- 2 \skp{\vec{f}}{\partial_t \m_1}{\Omega}
		\Big) \d{t}.
	\end{split}
	\end{equation}
	First, we expand the third term in the integral with $\skp{\vec{f}}{\w^\varepsilon}{\Omega}$ to obtain that
	\begin{equation}
	\label{eq:proof:term2}
		- 2 \skp{\vec{f}}{\partial_t \m_1}{\Omega}
		=
		- \skp{\vec{f}}{\partial_t \m_1}{\Omega}
		- \skp{\vec{f}}{\partial_t \m_2^\varepsilon}{\Omega}
		+ \skp{\vec{f}}{\partial_t \w^\varepsilon}{\Omega}.
	\end{equation}
	The first terms on the right hand sides of~\eqref{eq:proof:limes-estimate}--\eqref{eq:proof:term2} can be treated with \eqref{eq:proof:Kcorrespondence} and the strong form of LLG~\eqref{eq:strong-reformulation} to obtain that
	\begin{align}
	\nonumber
		\psi[\m_2^\varepsilon, \partial_t \m_1]
		&- \skp{\vec{f}}{\partial_t \m_1}{\Omega}
		\stackrel{\mathclap{\eqref{eq:proof:Kcorrespondence}}}{=}
		\skp{\Psi[\m_2^\varepsilon] - \vec{f}}{\partial_t \m_1}{\Omega} + \skp{\nabla_\helical \m_2^\varepsilon \cdot \normalv}{\partial_t \m_1}{\partial \Omega} \\
	\label{eq:proof:term1}
		&\stackrel{\mathclap{\eqref{eq:strong-reformulation}}}{=}
		\skp{\Psi[\m_2^\varepsilon] - \vec{f}}{\m_1 \times (\Psi[\m_1] - \vec{f})}{\Omega}
		+ \skp{\nabla_\helical \m_2^\varepsilon \cdot \normalv}{\m_1 \times (\Psi[\m_1] - \vec{f})}{\partial \Omega}\\
	\nonumber
		&=
		\skp{\Psi[\w^\varepsilon]}{\m_1 \times (\Psi[\m_1] - \vec{f})}{\Omega}
		+ \skp{\nabla_\helical \m_2^\varepsilon \cdot \normalv}{\m_1 \times (\Psi[\m_1] - \vec{f})}{\partial \Omega}.
	\end{align}
	Similarly, the second terms on the right hand sides of~\eqref{eq:proof:limes-estimate}--\eqref{eq:proof:term2} can be combined by the weak form~\eqref{eq:weak-reformulation} for $\m_2$ and $\vec{\varphi} = \Psi[\m_1] - \vec{f}$.
	We get that
	\begin{align}
	\nonumber
		\lim\limits_{\varepsilon \to 0} &\int_{0}^{t} \skp{\partial_t \m_2^\varepsilon}{\Psi[\m_1] - \vec{f}}{\Omega} \d{t} \\
	\label{eq:proof:term3}
		&\stackrel{\mathclap{\eqref{eq:weak-reformulation}}}{=}
		\lim\limits_{\varepsilon \to 0} \int_{0}^{t} \Big( \psi[\m_2^\varepsilon, (\Psi[\m_1] - \vec{f}) \times \m_2^\varepsilon] - \skp{\vec{f}}{(\Psi[\m_1] - \vec{f}) \times \m_2^\varepsilon}{\Omega} \Big) \d{t}\\
	\nonumber
		&\stackrel{\mathclap{\eqref{eq:proof:Kcorrespondence}}}{=}
		\lim\limits_{\varepsilon \to 0} \int_{0}^{t} \!\!\! \Big( \! \skp{\Psi[\m_2^\varepsilon] - \vec{f}}{(\Psi[\m_1] - \vec{f}) \! \times \! \m_2^\varepsilon}{\Omega}
		+ \skp{\nabla_\helical \m_2^\varepsilon \cdot \normalv}{(\Psi[\m_1] - \vec{f}) \! \times \! \m_2^\varepsilon}{\partial \Omega} \! \Big) \! \d{t}\\
	\nonumber
		&=
		\lim\limits_{\varepsilon \to 0} \int_{0}^{t} \Big( \skp{\Psi[\w^\varepsilon]}{(\Psi[\m_1] - \vec{f}) \! \times \! \m_2^\varepsilon}{\Omega}
		+ \skp{\nabla_\helical \m_2^\varepsilon \cdot \normalv}{(\Psi[\m_1] - \vec{f}) \! \times \! \m_2^\varepsilon}{\partial \Omega} \Big) \d{t},
	\end{align}
	where the last equality sign comes from $\Psi[\w^\varepsilon] = (\Psi[\m_2^\varepsilon] - \vec{f}) - (\Psi[\m_1] - \vec{f})$ and properties of the cross product.
	
	Using that $\m_1$ is a strong solution an thus satisfies $\nabla_\helical \m_1 \cdot \normalv = 0$, and combining the boundary terms from \eqref{eq:proof:term1}--\eqref{eq:proof:term3} yields that
	\begin{equation}
	\label{eq:proof:term4}
	\begin{split}
		&\skp{\nabla_\helical \m_2^\varepsilon \cdot \normalv}{\m_1 \times (\Psi[\m_1] - \vec{f})}{\partial \Omega}
		+ \skp{\nabla_\helical \m_2^\varepsilon \cdot \normalv}{(\Psi[\m_1] - \vec{f}) \! \times \! \m_2^\varepsilon}{\partial \Omega} \\
		& \quad  =
		- \skp{\nabla_\helical \m_2^\varepsilon \cdot \normalv}{\w^\varepsilon \times (\Psi[\m_1] - \vec{f})}{\partial \Omega}
		\stackrel{\eqref{eq:helical-boundary-condition}}{=}
		- \skp{\nabla_\helical \w^\varepsilon \cdot \normalv}{\w^\varepsilon \times (\Psi[\m_1] - \vec{f})}{\partial \Omega}.
	\end{split}
	\end{equation}
	By combining~\eqref{eq:proof:limes-estimate}--\eqref{eq:proof:term4}, we get that
	\begin{align*}
		&\EE_\helical[\w(t), \vec{f}(t)] + \DD[\w, \vec{f}](t) \\
		& \leq
		- \lim\limits_{\varepsilon \to 0}
		\int_{0}^{t} \!\! \Big( \skp{\Psi[\w^\varepsilon]}{\m_1 \! \times \! (\Psi[\m_1] - \vec{f})}{\Omega}
		+ \skp{\Psi[\w^\varepsilon]}{(\Psi[\m_1] - \vec{f}) \! \times \! \m_2^\varepsilon}{\Omega} \\
		& \hspace{150pt}
		+ \skp{\vec{f}}{\partial_t \w^\varepsilon}{\Omega}
		- \skp{\nabla_\helical \w^\varepsilon \cdot \normalv}{\w^\varepsilon \times (\Psi[\m_1] - \vec{f})}{\partial \Omega} \Big) \d{t} \\
		& =
		\lim\limits_{\varepsilon \to 0}
		\int_{0}^{t} \!\!\! \Big( \! \skp{\Psi[\w^\varepsilon]}{\w^\varepsilon \! \times \! (\Psi[\m_1] - \vec{f})}{\Omega}
		- \skp{\vec{f}}{\partial_t \w^\varepsilon}{\Omega}
		+ \skp{\nabla_\helical \w^\varepsilon \cdot \normalv}{\w^\varepsilon \! \times \! (\Psi[\m_1] - \vec{f})}{\partial \Omega} \! \Big) \! \d{t}\\
		& \stackrel{\mathclap{\eqref{eq:proof:Kcorrespondence}}}{=}
		\int_{0}^{t} \Big( \psi[\w, \w \times (\Psi[\m_1] - \vec{f})]
		- \skp{\vec{f}}{\partial_t \w}{\Omega} \Big) \d{t}.
	\end{align*}
	Integrating $\int_{0}^{t} \skp{\vec{f}}{\partial_t \w}{\Omega} \d{t}$ by parts in time and using that $\w(0) = 0$, we obtain that
	\begin{equation*}
		- \int_{0}^{t} \skp{\vec{f}}{\partial_t \w}{\Omega} \d{t}
		=
		- \skp{\vec{f}(t)}{\w(t)}{\Omega}
		+ \int_{0}^{t} \skp{\partial_t \vec{f}}{\w}{\Omega} \d{t},
	\end{equation*}
	and thus~\eqref{eq:proof:estimate2}.
\end{proof}

We have now collected all tools to prove our main theorem.

\begin{proof}[Proof of Theorem~\ref{theorem:main}]
	Starting from the estimate~\eqref{eq:proof:estimate2}, we expand all terms.
	This yields that
	\begin{align*}
		\frac{1}{2} \norm{\nabla_\helical \w}{\Omega}^2
		& - \frac{1}{2} \skp{\w(t)}{\boldsymbol{\pi}(\w(t))}{\Omega}
		+ \alpha \int_0^t \norm{\partial_t \w}{\Omega}^2 \d{t}\\
		& \leq
		\int_0^t \Big(
			\alpha \skp{\partial_t \w}{\w \times (\Psi[\m_1] - \vec{f})}{\Omega}
			+ \skp{\nabla_\helical \w}{\w \times \nabla_\helical (\Psi[\m_1] - \vec{f})}{\Omega} \\
			& \quad - \skp{\boldsymbol{\pi}(\w)}{\w \times (\Psi[\m_1] - \vec{f})}{\Omega}
		\Big) \d{t}.
	\end{align*}
	We note that this is \eqref{eq:strong:expanded-estimate} integrated in time (since $\w(0) = \vec{0}$).
	With Remark~\ref{rem:smoothness}, we can follow the proof of Theorem~\ref{theorem:strong-strong} line by line to obtain $\int_0^T \alpha \norm[]{ \partial_t \w }{\Omega}^2 \d{t} = 0$.
	This concludes the proof
\end{proof}

\textbf{Acknowledgement:} The authors like to express their gratitude to the anonymous referee for pointing out 
an issue with the presentation of the lower-order terms in an earlier version of the 
manuscript.

\bibliographystyle{alpha}
\bibliography{literature}

\end{document}